\documentclass[amsthm]{monsky2009} 

\usepackage{graphicx}

\usepackage{amssymb}
\usepackage{amsmath}

\usepackage{bm}

\newtheorem*{corollary*}{Corollary}
\newtheorem*{remark*}{Remark}
\newtheorem*{comments*}{Comments}

\newtheorem{definition}{Definition}[section]
\newtheorem{theorem}[definition]{Theorem}
\newtheorem{lemma}[definition]{Lemma}

\newtheorem{corollary}[definition]{Corollary}
\newtheorem{remark}[definition]{Remark}

\newcommand{\jacobi}[2]{\ensuremath{\left(\frac{#1}{#2}\right)}}

\newcommand{\consti}{\ensuremath{\mathrm{i}}}

\newcommand{\pr}{\ensuremath{\mathit{pr}}}

\newcommand{\newo}{\ensuremath{\mathcal{O}}}

\newcommand{\newz}{\ensuremath{\mathbb{Z}}}

\begin{document}

\begin{frontmatter}






\title{Disquisitiones Arithmetic\ae\ and online sequence $A$108345}
\author{Paul Monsky}

\address{Brandeis University, Waltham MA  02454-9110, USA\\  monsky@brandeis.edu }

\begin{abstract}
Let $g$ be the element $\sum_{n\ge 0}x^{n^{2}}$ of $A=\newz/2[[x]]$, and $B$ consist of all $n$ for which the coefficient of $x^{n}$ in $\frac{1}{g}$ is 1. (The elements of $B$ are the entries 0, 1, 2, 3, 5, 7, 8, 9, 13, \ldots in $A$108345; see \cite{3}.)  In \cite{1} it is shown that the (upper) density of $B$ is $\le\frac{1}{4}$, and it is conjectured that $B$ has density 0. This note uses results of Gauss on sums of 3 squares to show that the subset of $B$ consisting of $n\not\equiv 15\pod{16}$ has density 0. The final section gives some computer calculations, made by Kevin O'Bryant, indicating that, pace \cite{1}, $B$ has density $\frac{1}{32}$.

\begin{comments*}
The note is drawn from my answers, on Mathoverflow, to questions asked by O'Bryant and me.
\end{comments*}
\end{abstract}
\maketitle

\end{frontmatter}


\section{Introduction}
\label{introduction}

I begin with simple derivations of some results from \cite{1}. Let $g$ be the element $1+x+x^{4}+x^{9}+\cdots$ of $A=\newz/2 [[x]]$. Write $\frac{1}{g}$ as $\sum b_{i}x^{i}$ with the $b_{i}$ in $\newz/2$, and let $B$ consist of all $n$ with $b_{n}=1$.

\begin{theorem}
\label{theorem1.1}
If $n$ is even, $n$ is in $B$ if and only if $\frac{n}{2}$ is a square.
\end{theorem}

\begin{proof}
Let $R\subset A$ be $\newz/2 [[x]]$. As $R$-module, $A$ is the direct sum of $R$ and $xR$. Let $\pr: A\rightarrow R$ be the $R$-linear map which is the identity on $R$ and sends $xR$ to 0. Since $g^{2}$ is in $R$, so is $\frac{1}{g^{2}}$. Now $\pr(g)=1+x^{4}+x^{16}+x^{36}+\cdots = g^{4}$.  So $\pr\left(\frac{1}{g}\right)=\frac{1}{g^{2}}\pr(g)=g^{2}$. This is precisely the statement of the theorem.
\end{proof}

\begin{theorem}
\label{theorem1.2}
If $n\equiv 1\pod{4}$, $n$ is in $B$ if and only if the number of ways of writing $n$ as $(\mathit{square})+4(\mathit{square})$ is odd.
\end{theorem}

\begin{proof}
$\frac{1}{g}=g\cdot\frac{1}{g^{2}}$. So the coefficient of $x^{n}$ in $\frac{1}{g}$ is the number of ways, modulo 2, of writing $n$ as $(\mathit{square})+2k$ with $k$ in $B$. Since $n\equiv 1\pod{4}$, the square is also $\equiv 1\pod{4}$, and $k$ is even. Now use Theorem \ref{theorem1.1}.
\end{proof}

\begin{theorem}
\label{theorem1.3}
The number of $n$ in $B$ that are $\le x$ and $\not\equiv 3\pod{4}$ is $O(x/\!\log(x))$.
\end{theorem}

\begin{proof}
In view of Theorem \ref{theorem1.1} we may restrict our attention to $n$ that are $\equiv 1\pod{4}$ (and that are not squares). If such an $n$ is $s_{1}+4s_{2}$ then $\sqrt{s_{1}}+2\consti\sqrt{s_{2}}$ and $\sqrt{s_{1}}-2\consti\sqrt{s_{2}}$ generate ideals of norm $n$ in $\newz[i]$; since $n$ is not a square, these two ideals are distinct. Since every ideal of norm $n$ comes from exactly one decomposition of $n$ as $(\mathit{square})+4(\mathit{square})$, the number of decompositions of $n$ is $\frac{1}{2}(\mathit{the\ number\ of\ ideals\ of\ norm\ n})$. Standard facts about $\newz[i]$ tell us that this number is odd only when $n$ is the product of a square by a prime $\equiv 1\pod{4}$. Now use the fact that $\pi(x)=O(x/\!\log(x))$.
\end{proof}

\begin{theorem}
\label{theorem1.4}
If $n\equiv 3\pod{8}$, $n$ is in $B$ if and only if the number of ways of writing $n$ as $(\mathit{square})+2(\mathit{square})+8(\mathit{square})$ is odd.
\end{theorem}

\begin{proof}
$\frac{1}{g}=g\cdot g^{2}\cdot\frac{1}{g^{4}}$. So the coefficient of $x^{n}$ in $\frac{1}{g}$ is the number of ways, modulo 2, of writing $n$ as $(\mathit{square})+2(\mathit{square})+4k$ with $k$ in $B$. Since $n\equiv 3\pod{8}$, congruences mod 8 show that $k$ is even, and we use Theorem \ref{theorem1.1}.
\end{proof}

\section{A density result for $n\equiv 3\pod{8}$}

\begin{lemma}
\label{lemma2.1}
Suppose $n\equiv 3\pod{8}$. Let $R_{1}$ and $R_{2}$ be the number of ways of writing $n$ as $(\mathit{square})+(\mathit{square})+(\mathit{square})$ and as $(\mathit({square}))+2(\mathit{square})$.  If 4 divides $R_{1}$ and $R_{2}$, then $n$ is not in $B$.
\end{lemma}

\begin{proof}
In view of Theorem \ref{theorem1.4} it suffices to show that $R_{1}+R_{2}$ is twice the number of ways of writing $n$ as $(\mathit{square})+2(\mathit{square})+8(\mathit{square})$. Suppose $n=s_{1}+s_{2}+s_{3}$ with the $s_{i}$ squares. The $s_{i}$ are odd.  Let $r_{2}$ and $r_{3}$ be square roots of $s_{2}$ and $s_{3}$ with $r_{2}\equiv r_{3}\pod{4}$. Then $n=s_{1}+2\left(\frac{r_{2}+r_{3}}{2}\right)^{2}+8\left(\frac{r_{2}-r_{3}}{4}\right)^{2}=(\mathit{square})+2(\mathit{square})+8(\mathit{square})$, and replacing $r_{2}$ and $r_{3}$ by $-r_{2}$ and $-r_{3}$ gives the same decomposition.  It's easy to see that one gets every decomposition $n=t_{1}+2t_{2}+8t_{3}$ with the $t_{i}$ squares from some triple $(s_{1},s_{2},s_{3})$ in this way. Furthermore if $(s_{1},s_{2},s_{3})\rightarrow(t_{1},t_{2},t_{3})$, then $(s_{1},s_{3},s_{2})\rightarrow$ the same $(t_{1},t_{2},t_{3})$.  It follows that the fiber over a fixed $(t_{1},t_{2},t_{3})$ consists of 2 elements except at those points where $t_{3}=0$. But such a point corresponds to a decomposition of $n$ as $(\mathit{square})+2(\mathit{square})$.
\end{proof}

\begin{lemma}
\label{lemma2.2}
Suppose $n\equiv 3\pod{8}$ and is divisible by 3 or more different primes. Then the number of ways of writing $n$ primitively as $(\mathit{square})+(\mathit{square})+(\mathit{square})$ is divisible by 4.
\end{lemma}

\begin{proof}
Let $\newo=\newz\left[\frac{1+\sqrt{-n}}{2}\right]$. A result of Gauss, \cite{2}, put into modern language, is that the number of primitive representations of $n$ by the form $x^{2}+y^{2}+z^{2}$ is $24\cdot(\mathit{the\ number\ of\ invertible\ ideal\ classes\ in\ }O)$. So the number of ways of writing $n$ primitively as $(\mathit{square})+(\mathit{square})+(\mathit{square})$ is $3\cdot(\mathit{the\ number\ of\ invertible\ ideal\ classes})$, and it suffices to show that 4 divides this number. Now Gauss developed a genus theory for binary quadratic forms which tells us that the group of invertible ideal classes maps onto a product of $m-1$ copies of $\newz/2$, where $m$ is the number of different primes dividing $n$. Since $m\ge 3$ we're done.
\end{proof}

\begin{theorem}
\label{theorem2.3}
If $n\equiv 3\pod{8}$ and there are 3 or more primes that occur to odd exponent in the prime factorization of $n$, then $n$ is not in $B$.
\end{theorem}

\begin{proof}
By Lemma \ref{lemma2.2}, whenever $a^{2}$ divides $n$, the number of ways of writing $n/a^{2}$ primitively as $(\mathit{square})+(\mathit{square})+(\mathit{square})$ is divisible by 4. Summing over $a$ we find that 4 divides $R_{1}$.  Furthermore, by Lemma \ref{lemma3.3}, $2R_{2}$ is the number of ideals of norm $n$ in $\newz\left[\sqrt{-2}\right]$. This number is $\sum\jacobi{-2}{d}$ where $\jacobi{}{}$ is the Jacobi symbol, and $d$ runs over the divisors of $n$. Since $\jacobi{}{}$ is multiplicative, the sum is a product of integer factors, one coming form each prime dividing $n$. Also, a prime having odd exponent in the factorization contributes an even factor. Since there are at least 3 such primes, 8 divides $2R_{1}$, 4 divides $R_{1}$, and we use Lemma \ref{lemma2.1}.
\end{proof}

\begin{theorem}
\label{theorem2.4}
The number of $n$ in $B$ that are $\le x$ and $\equiv 3\pod{8}$ is \linebreak $O\left(x\log\log(x)/\log(x)\right)$.
\end{theorem}

\begin{proof}
Let $\pi_{2}(x)$ be the number of $n\le x$ that are a product of 2 primes. It's well-known that $\pi_{2}(x)$ is $O\left(x\log\log(x)/\log(x)\right)$. By Theorem \ref{theorem2.3} an element of $B$ that is $\equiv 3\pod{8}$ is either the product of a single prime and a square, or of two primes and a square. The result follows easily.
\end{proof}

\section{A density result for $\bm{n\equiv 7\pod{16}}$}
\label{section3}

For $n\equiv 7\pod{16}$ we show that $n$ is in $B$ if and only if the number of ways to write $2n$ as $(\mathit{square})+(\mathit{square})+(\mathit{square})$ is $\equiv 2\pod{4}$, and arguing as in the last section, prove the analogue to Theorem \ref{theorem2.4}.

\begin{lemma}
\label{lemma3.1}
If $n\equiv 1\pod{8}$ then the number of ideals $U$ of norm $n$ in $\newz\left[\sqrt{-2}\right]$ is congruent mod 4 to the number of ideals $V$ of norm $n$ in $\newz[\consti]$ unless $n=A^{2}$ with $A\equiv \pm 3\pod{8}$.
\end{lemma}

\begin{proof}
$U=\sum\jacobi{-2}{d}$ and $V=\sum\jacobi{-1}{d}$ where the sums are over the divisors of $n$. Since $\jacobi{}{}$ is multiplicative, $U$ (resp.\ $V$) is a product of contributions, one for each prime dividing $n$. A contribution is even if the prime occurs to odd exponent in the factorization of $n$, and is odd otherwise. In particular if 2 or more $p$ appear to odd exponent, then 4 divides $U$ and $V$. Next suppose there is exactly one prime $p$ occurring with odd exponent and that the exponent is $c$. Since $n\equiv 1\pod{8}$, $p\equiv 1\pod{8}$, and $\jacobi{-2}{p}=\jacobi{-1}{p}=1$. So $p$ makes a contribution of $c+1$ both to $U$ and to $V$. Since all the other contribution are odd, $U\equiv V\equiv 0\pod{4}$ when $c\equiv 3\pod{4}$, and $U\equiv V\equiv 2\pod{4}$ when $c\equiv 1\pod{4}$.

It remains to analyze the case $n=A^{2}$. In this case $U$ and $V$ are odd, and we are reduced to showing: if $A\equiv \pm 1\pod{8}$ then $UV\equiv 1\pod{4}$, while if $A\equiv \pm 3\pod{8}$, then $UV\equiv 3\pod{4}$. Consider $UV$ as an element of the multiplicative group $\{1,3\}$ of $\newz/4$. $UV$ is a product of contributions, one for each prime dividing $A$. A $p\equiv \pm 1\pod{8}$ makes the same contribution to $U$ as to $V$ and so does not contribute to the product.  If on the other hand $p\equiv\pm 3\pod{8}$ and has exponent $c$ in the factorization of $A$ then the contribution it makes to $UV$ is $(2c+1)\cdot 1$ when $p\equiv 3\pod{8}$ and $1\cdot(2c+1)$ when $p\equiv -3\pod{8}$. In other words the contribution is $-1$ precisely when $c$ is odd. This tells us that $UV\equiv 1\pod{4}$ when the number of primes $\equiv \pm 3\pod{8}$ with odd exponent in the factorization of $A$ is even, and that $UV\equiv 3\pod{4}$ when this number is odd. But in the first case $A\equiv\pm 1\pod{8}$, while in the second $A\equiv\pm 3\pod{8}$.
\end{proof}

\begin{definition}
\label{definition3.2}
Suppose $n$ is odd. $U_{1}$ is the number of ways of writing $n$ as $(\mathit{square})+2(\mathit{square})$ while $U_{2}$ is the number of ways of writing $n$ as $(\mathit{square})+4(\mathit{square})$.
\end{definition}

\begin{lemma}
\label{lemma3.3}
The number of ideals $U$ of $\newz\left[\sqrt{-2}\right]$ of norm $n$ is $2U_{1}-1$ when $n$ is a square and $2U_{1}$ otherwise. The number of ideals $V$ of $\newz[\consti]$  of norm $n$ is $2V_{1}-1$ when $n$ is a square and $2V_{1}$ otherwise.
\end{lemma}

\begin{proof}
Suppose $n=s_{1}+2s_{2}$ with $s_{1}$ and $s_{2}$ squares. Then $\sqrt{s_{1}}+\sqrt{-2}\sqrt{s_{2}}$ and $\sqrt{s_{1}}-\sqrt{-2}\sqrt{s_{2}}$ generate ideals of norm $n$ in $\newz\left[\sqrt{-2}\right]$. These 2 ideals are distinct except when $n$ is a square and $s_{2}=0$. Also every ideal of norm $n$ comes from exactly one such decomposition of $n$. This gives the first result and the proof of the second is similar.
\end{proof}

Lemmas \ref{lemma3.1} and \ref{lemma3.3} immediately give:

\begin{lemma}
\label{lemma3.4}
If $n\equiv 1\pod{16}$, then $U_{1}\equiv V_{1}\pod{2}$.
\end{lemma}

\begin{lemma}
\label{lemma3.5}
If $n\equiv 1\pod{16}$, then the coefficient of $x^{n}$ in $\frac{1}{g^{7}}$ is 1 if and only if $n$ is a square.
\end{lemma}

\begin{proof}
Since $n\equiv 1\pod{8}$, the number of ways $U_{1}$ of writing $n$ as $(\mathit{square})+2(\mathit{square})$ is the number of ways of writing $n$ as $(\mathit{square})+8(\mathit{square})$. So the image of $U_{1}$ in $\newz/2$ is the coefficient of $x^{n}$ in $g\cdot g^{8}=g^{9}$. Similarly, the image of $V_{1}$ in $\newz/2$ is the coefficient of $x^{n}$ in $g\cdot g^{16}=g^{17}$. Lemma \ref{lemma3.4} then tells us that for $n\equiv 1\pod{16}$ the coefficients of $x^{n}$ in $g^{9}$ and in $g^{17}$ are equal.

Now let $S\subset A$ be $\newz/2 [[x^{16}]]$. As $S$-module $A$ is the direct sum of the $x^{j}S$, $0\le j\le 15$. Let $\pr:A\rightarrow xS$ be the $S$-linear map that is the identity on $xS$ and 0 on the other summands. The last paragraph tells us that $\pr(g^{9})=\pr(g^{17})$. Since $\frac{1}{g^{16}}$ is in $S$, $\pr\left(\frac{1}{g^{7}}\right)=\pr(g)$. But as $n\equiv 1\pod{16}$, the coefficient of $x^{n}$ in $\pr(g)$ is the coefficient of $x^{n}$ in $g$, giving the result.
\end{proof}

\begin{theorem}
\label{theorem3.6}
If $n\equiv 7\pod{16}$ then $n$ is in $B$ if and only if the number of ways of writing $n$ as $(\mathit{square})+2(\mathit{square})+4(\mathit{square})$ is odd.
\end{theorem}

\begin{proof}
$\frac{1}{g}=g^{2}\cdot g^{4}\cdot \frac{1}{g^{7}}$. So the coefficient of $x^{n}$ in $\frac{1}{g}$ is the number of ways, modulo 2, of writing $n$ as $2(\mathit{square})+4(\mathit{square})+k$ with the coefficient of $x^{k}$ in $\frac{1}{g^{7}}$ equal to 1. Suppose we have such a representation of $n$. Then $k$ is odd. Since $\frac{1}{g^{7}}=\frac{g}{g^{8}}$ it follows that $k\equiv 1\pod{8}$ A congruence mod 16 argument using the fact that $n\equiv 7\pod{16}$ shows that $k\equiv 1\pod{16}$, and Lemma \ref{lemma3.5} tells us that $k$ is a square. Conversely suppose $n=2(\mathit{square})+4(\mathit{square})+k$, where $k$ is a square. Then $k\equiv 1\pod{8}$ and our congruence mod 16 argument tells us that $k\equiv 1\pod{16}$. By Lemma \ref{lemma3.5}, the coefficient of $x^{k}$ in $\frac{1}{g^{7}}$ is 1, and this completes the proof.
\end{proof}

\begin{lemma}
\label{lemma3.7}
Let $R_{3}$ be the number of ways of writing $2n$ as $(\mathit{square})+(\mathit{square})+(\mathit{square})$. Then if $n\equiv 7\pod{8}$, $R_{3}=6\cdot(\mathit{the\ number\ of\ ways\ of\ writing}$ $n$ $\mathit{ as\ }(\mathit{square})+2(\mathit{square})+4(\mathit{square}))$.
\end{lemma}

\begin{proof}
Suppose $2n=s_{1}+s_{2}+s_{3}$ with the $s_{i}$ squares. A congruence mod 16 argument shows that the $s_{i}$, in some order, are $\equiv$ 1, 4 and 9 mod 16. So $R_{3}=6\cdot(\mathit{the\ number\ of\ ways\ of\ writing\ } 2n\ \mathit{as\ } s_{1}+s_{2}+s_{3}\ \mathit{with\ the\ } s_{i}\ \mathit{squares}, s_{1}\equiv 1\pod{16}, s_{2}\equiv 4\pod{16}, s_{3}\equiv 9\pod{16})$. Suppose we have such a representation. Then we can choose square roots of $s_{1}$ and $s_{3}$ congruent to 1 and 5 respectively mod 8. Then $n=\left(\frac{\sqrt{s_{1}}+\sqrt{s_{3}}}{2}\right)^{2}+2\left(\frac{s_{2}}{4}\right)+4\left(\frac{\sqrt{s_{1}}-\sqrt{s_{3}}}{4}\right)^{2}=\mathit{(square)}+2\mathit{(square)}+4\mathit{(square)}$. Conversely suppose $n=t_{1}+2t_{2}+4t_{3}$ with the $t_{i}$ squares. Then the $t_{i}$ are odd. Choose square roots of $t_{1}$ and $t_{3}$ that are $\equiv 1\pod{4}$.  Then $2n=\left(2\sqrt{t_{3}}-\sqrt{t_{1}}\right)^{2}+4t_{2}+\left(2\sqrt{t_{3}}+\sqrt{t_{1}}\right)^{2}$, and the three squares appearing in this decomposition are, in order, congruent mod 16 to 1,4 and 9. In this way we get a 1--1 correspondence that establishes the result.
\end{proof}

Combining Theorem \ref{theorem3.6} and Lemma \ref{lemma3.7} we get:

\begin{theorem}
\label{theorem3.8}
An $n\equiv 7\pod{16}$ is in $B$ if and only if the $R_{3}$ of Lemma \ref{lemma3.7} is $\equiv 2\pod{4}$.
\end{theorem}

\begin{lemma}
\label{lemma3.9}
Suppose $n\equiv 7\pod{8}$ and is divisible by 3 or more different primes. Then the number of ways of writing $2n$ primitively as $\mathit{(square)}+\mathit{(square)}+\mathit{(square)}$ is divisible by 4.
\end{lemma}

\begin{proof}
Let $\newo=\newz\left[\sqrt{-2n}\right]$. When we write $2n$ as $\mathit{(square)}+\mathit{(square)}+\mathit{(square)}$, the summands, being $\equiv$ 1, 4 and 9 mod 16 are non-zero and distinct. So the number we're talking about is $\frac{1}{8}\cdot(\mathit{the\ number\ of\ primitive\ representations\ of\ }$ $2n\ \mathit{by\ the\ form\ } x^{2}+y^{2}+z^{2})$. In \cite{2} Gauss showed that this (in modern language) is $\frac{1}{8}\cdot12\cdot(\mathit{the\ number\ of\ invertible\ ideal\ classes\ in\ }\newo)$. Let $m$ be the number of different primes dividing $2n$. Gauss' genus theory tells us that the group of invertible ideal classes maps onto a product of $m-1$ copies of $\newz/2$. Since $m\ge 4$ we're done.
\end{proof}

\begin{corollary}
\label{corollary3.10}
If $n\equiv 7\pod{8}$ and 3 or more different primes occur to odd exponent in the factorization of $n$, then the $R_{3}$ of Lemma \ref{lemma3.7} is divisible by 4.
\end{corollary}

\begin{proof}
For $a^{2}$ dividing $2n$, Lemma \ref{lemma3.9} shows that the number of ways of writing $2n/a^{2}$ primitively as $\mathit{(square)}+\mathit{(square)}+\mathit{(square)}$ is a multiple of 4. Summing over $a$ gives the result.
\end{proof}

\begin{theorem}
\label{theorem3.11}
If $n\equiv 7\pod{16}$ and 3 or more primes occur to odd exponent in the factorization of $n$ then $n$ is not in $B$. Furthermore the number of $n$ in $B$ that are $\le x$ and $\equiv 7\pod{16}$ is $O(x \log \log(x)/\log(x))$.
\end{theorem}

\begin{proof}
Theorem \ref{theorem3.8} and Corollary \ref{corollary3.10} give the first result, and we argue as in Theorem \ref{theorem2.4} to get the second.
\end{proof}

Combining Theorems \ref{theorem1.3}, \ref{theorem2.4} and \ref{theorem3.11} we get:

\begin{theorem}
\label{theorem3.12}
The number of $n$ in $B$ that are $\le x$ and $\not\equiv 15\pod{16}$ is $O(x \log \log(x)/\log(x))$. In particular the upper density of $B$ is $\le\frac{1}{16}$.
\end{theorem}

Can one go further? A hope would be to find extensions of Theorems \ref{theorem1.1}, \ref{theorem1.2} and \ref{theorem1.4} of this note that hold for $n\equiv 7\pod{16}$, $n\equiv 15\pod{32}$, $n\equiv 31\pod{64}$, \ldots . The authors of \cite{1} claim that such extensions exist, but apart from $n\equiv 7\pod{16}$, treated in this section, this seems unlikely. (The formulas they propose are incorrect.)  There seems to be no theoretical evidence supporting the proposition that the $n\equiv 15\pod{16}$ that lie in $B$ form a set of density 0. As we'll see in the next section the empirical evidence supports a quite different proposition.

\section{Computer evidence when $\bm{n\equiv 15\pod{16}}$}

Suppose $x$ is in $N$. There evidently are $x$ positive integers that are $\le 16x$ and $\equiv 15\pod{16}$. Let $\beta=\beta(x)$ be the number of these integers that are in $B$. Virtually nothing is known about the asymptotic growth of $\beta$. But Kevin O'Bryant has calculated $\beta$ for $x\le2^{19}$, and his calculations show, for example:

\begin{enumerate}
\item[(1)] If $x=2^{16}$, the numbers of elements of $B$ that are $\equiv 15\pod{16}$ and lie in $[0,16x],[16x,32x],\ldots,[112x,128x]$, are given respectively by $\frac{x}{2}+13$, $\frac{x}{2}+94$, $\frac{x}{2}-231$, $\frac{x}{2}+207$, $\frac{x}{2}-120$, $\frac{x}{2}+14$, $\frac{x}{2}-270$ and $\frac{x}{2}+7$.

\item[(2)] Suppose $x\le 2^{19}$ and is divisible by $2^{10}$. Then $\beta = \frac{x}{2}+\alpha\sqrt{x}$ with $-1.1<\alpha < .58$. (The minimum of $\alpha$ is attained at $5\cdot 2^{10}$, and the maximum at $37\cdot 2^{10}$.)
\end{enumerate}

This provides evidence for the following ``15 mod 16 conjecture'': For every $\rho >\frac{1}{2}$, $\beta =\frac{x}{2}+O\left(x^{\rho}\right)$.

Note that if the conjecture holds then Theorem \ref{theorem3.12} shows that $B$ has density $\frac{1}{32}$.

\begin{remark}
\label{remark4.1}
There is a related much studied problem. Let $g^{*}$ in $\newz/2[[x]]$ be $1+x+x^{2}+x^{5}+x^{7}+\cdots$ where the exponents are the generalized pentagonal numbers. Just as we used $\frac{1}{g}$ to define $B$ we can use $\frac{1}{g^{*}}$ to define a set $B^{*}$. (A famous result of Euler says that $B^{*}$ consists of all $n$ for which the number of partitions, $p(n)$, of $n$ is odd.)  Let $\beta^{*}=\beta^{*}(x)$ be the number of elements of $B^{*}$ that are $\le x$. Despite extensive study only very weak results about the asymptotic growth of $\beta^{*}$ have been proved. But Parkin and Shanks \cite{4}, on the basis of computer calculations, conjectured that for every $\rho >\frac{1}{2}$, $\beta=\frac{x}{2}+O\left(x^{\rho}\right)$. The resistance of this conjecture to attack suggests however that any proof of our 15 mod 16 conjecture is far off.
\end{remark}


\label{}



\end{document}